\documentclass[
final
, nomarks
]{dmtcs-episciences}

\usepackage[utf8]{inputenc}
\usepackage{subfigure}
\usepackage{amsmath}
\usepackage{cases}

\usepackage[round]{natbib}

\newtheorem{thm}{Theorem}
\newtheorem{coro}[thm]{Corollary}

\newtheorem{fact}[thm]{Fact}
\newtheorem{lem}[thm]{Lemma}

\author[Aijun Yi et al.]{Aijun Yi
  \and Zhidan Luo\thanks{Corresponding author. Research supported by the National Natural Science Foundation of China (No.12401449), the Hainan Provincial Natural Science Foundation of China (No.125QN209) and the Research Foundation Project of Hainan University (No. KYQD(ZR)-23155).}}
\title[The $s$-chromatic Ramsey number for stars]{The $s$-chromatic Ramsey number for stars}
\affiliation{
  School of Mathematics and Statistics,\\
  Hainan University, Haikou, P. R. China}
\keywords{$s$-chromatic Ramsey number, star-critical $s$-chromatic Ramsey number, star}
\begin{document}
% This is only used if you are compiling for a volume before vol 25
% \publicationdetails{VOL}{2015}{ISS}{NUM}{SUBM}
% This is the new form of collecting the data, starting with vol 25
%\publicationdata
%{vol. 25:3 special issue for main purpose}
%{2022}
%{1}
%{10.46298/dmtcs.10472}
%%{1998-10-14; 1998-10-14; 2002-07-19; 2014-02-05; 2015-09-09; 2022-12-25}
%%{2022-12-3}
%{2022-12-3; None}
%{2023-1-1}

\publicationdata{vol. 28:3}{2026}{5}{10.46298/dmtcs.17152}{2025-12-19; 2025-12-19; 2026-06-19}{2026-07-20}

\maketitle
\begin{abstract}
  In 1977, Chung, Chung and Liu generalized the definition of the Ramsey number. They introduced the $s$-chromatic Ramsey number as follows. Let $1\leq s< t$ be integers and let $A_{1}, A_{2}, \dots, A_{c}$ be subsets with size $s$ of $[t]$, where $c= {t\choose s}$. For given graphs $G_{1}, G_{2}, \dots, G_{c}$, the {\it $s$-chromatic Ramsey number} $r^{s, t}(G_{1}, G_{2}, \dots, G_{c})$, is the minimum positive integer $N$ such that every $t$-coloring of $E(K_{N})$ yields a copy of $G_{i}$ whose edges are colored by colors in the color set $A_{i}$ for some $i\in [c]$. The {\it star-critical $s$-chromatic Ramsey number} $r_{*}^{s, t}(G_{1}, G_{2}, \dots, G_{c})$, is the minimum integer $\ell$ such that every $t$-coloring of the edges in $K_{N}- E(K_{1, N- 1- \ell})$ yields a copy of $G_{i}$ whose edges are colored by colors in the color set $A_{i}$ for some $i\in [c]$, where $N= r^{s, t}(G_{1}, G_{2}, \dots, G_{c})$. If $G_{1}= G_{2}= \dots= G_{c}= G$, then we simplify them to $r^{s, t}(G)$ (also called the {\it weakened Ramsey number}) and $r^{s, t}_{*}(G)$, respectively. In this paper, we determine all the values of $r^{s, t}(K_{1, m})$ and $r_{*}^{s, t}(K_{1, m})$, and part of the value of $r^{s, t}(K_{1, m_{1}}, K_{1, m_{2}}, \dots, K_{1, m_{c}})$.
\end{abstract}

\section{Introduction}
  All graphs in this paper are simple graphs. For a graph $G$, let $V(G)$ and $E(G)$ be the vertex set and the edge set of $G$, respectively. The order of $G$, $v(G)$, is $|V(G)|$. For given positive integer $t$, let $G_{1}, G_{2}, \dots, G_{t}$ be graphs and let $[t]$ be the set $\{1, 2, \dots, t\}$. In 1930, \cite{FR} introduced the {\it Ramsey number} $r(G_{1}, G_{2}, \dots, G_{t})$, which is the minimum integer $N$ such that every $t$-coloring of $E(K_{N})$ yields a monochromatic copy of $G_{i}$ in color $i$ for some $i\in [t]$. Moreover, if $G_{1}= G_{2}= \cdots= G_{t}= G$, then we simplify it to $r^{t}(G)$. Let $K_{1, m}$ be the star of order $m+ 1$. In 1972, \cite{H2} determined the value of $r(K_{1, m_{1}}, K_{1, m_{2}})$, and Burr and Roberts extended it.
  \begin{thm}[\cite{BR}]\label{theo1.1}
     If $m_{1}, m_{2}, \dots, m_{t}$ are integers larger than one, exactly $k$ of which are even, then
     $$r(K_{1, m_{1}}, K_{1, m_{2}}, \dots, K_{1, m_{t}})= \begin{cases}
       \sum_{i= 1}^{t} m_{i}- t+ 1, & k\geq 2\text{ is even},\\
       \sum_{i= 1}^{t} m_{i}- t+ 2, & \text{otherwise}.
     \end{cases}$$
   \end{thm} 
  \noindent For more results on the Ramsey number, we refer to the dynamic survey \cite{SR}.
	
  For a graph $G$ and a subgraph $H$ of $G$, let $G- H$ be the graph obtained from $G$ by removing a copy of $H$. In 2010, \cite{H3} introduced the {\it star-critical Ramsey number} $r_{*}(G_{1}, G_{2}, \dots, G_{t})$, which is the minimum integer $\ell$ such that every $t$-coloring of the edges in $K_{N}- E(K_{1, N- 1- \ell})$ yields a monochromatic copy of $G_{i}$ in color $i$ for some $i\in [t]$, where $N= r(G_{1}, G_{2}, \dots, G_{t})$. Moreover, if $G_{1}= G_{2}= \cdots= G_{t}= G$, then we simplify it to $r_{*}^{t}(G)$. Recently, Budden and DeJonge, and the second author determined the star-critical Ramsey number for stars.
  \begin{thm}[\cite{BD}, \cite{L}]\label{theo1.2}
    If $m_{1}, m_{2}, \dots, m_{t}$ are integers larger than one, exactly $k$ of which are even, then
     $$r_{*}(K_{1, m_{1}}, K_{1, m_{2}}, \dots, K_{1, m_{t}})= \begin{cases}
       \sum_{i= 1}^{t} m_{i}- t+ 1- k/2, & k\geq 2\text{is even},\\
       1, & \text{otherwise}.
     \end{cases}$$
  \end{thm}
  \noindent For more results on the star-critical Ramsey number, we refer to the survey \cite{J2} and the book \cite{M}.
	
  Let $1\leq s< t$ be integers and let $A_{1}, A_{2}, \dots, A_{c}$ be subsets with size $s$ of $[t]$, where $c= {t\choose s}$ is the binomial coefficient. In 1977, \cite{CCL} generalized the definition of the Ramsey number. They introduced the {\it $s$-chromatic Ramsey number} $r^{s, t}(G_{1}, G_{2}, \dots, G_{c})$, which is the minimum positive integer $N$ such that every $t$-coloring of $E(K_{N})$ yields a copy of $G_{i}$ whose edges are colored by colors in the color set $A_{i}$ for some $i\in [c]$. Moreover, if $G_{1}= G_{2}= \cdots= G_{c}= G$, then we simplify it to $r^{s, t}(G)$ (also called the {\it weakened Ramsey number} by \cite{HM}). Meanwhile, they determined the value of $r^{2, 3}(K_{1, m_{1}}, K_{1, m_{2}}, K_{1, m_{3}})$. In 2017, \cite{KO1} generalized it. They determined the value of $r^{t- 1, t}(K_{1, m_{1}}, K_{1, m_{2}}, \dots, K_{1, m_{t}})$ under some conditions. In 2022, Khamseh and Omidi also completely determined the value of $r^{t- 1, t}(K_{1, m})$.
  \begin{thm}[\cite{KO}]\label{theo1.3}
    Let $x= \lfloor(mt- 1)/(t- 1)\rfloor$ and $q= \lfloor x/t\rfloor$. Then 
    $$r^{t- 1, t}(K_{1, m})= \begin{cases}
      x, & \text{$x= tq+ 1$ and $x, q$ are odd},\\
      x+ 1, & \text{otherwise}.
    \end{cases}$$
  \end{thm}
  \noindent Meanwhile, they determined the value of $r^{t- 2, t}(K_{1, m})$.
  
  We can also generalize the star-critical Ramsey number to the star-critical $s$-chromatic Ramsey number. The {\it star-critical $s$-chromatic Ramsey number}, $r_{*}^{s, t}(G_{1}, G_{2}, \dots, G_{c})$, is the minimum integer $\ell$ such that every $t$-coloring of the edges in $K_{N}- E(K_{1, N- 1- \ell})$ yields a copy of $G_{i}$ whose edges are colored by colors in the color set $A_{i}$ for some $i\in [c]$, where $N= r^{s, t}(G_{1}, G_{2}, \dots, G_{c})$. Recently, Budden, Moun and Jakhar determined the value of $r_{*}^{t- 1, t}(K_{1, m})$.
  \begin{thm}[\cite{BMJ}]\label{theo1.4}
    Let $x= \lfloor (mt- 1)/(t- 1)\rfloor$ and $q= \lfloor x/t\rfloor$. Then 
    $$r_{*}^{t- 1, t}(K_{1, m})= \begin{cases}
      x- 1, & \text{$x= tq+ 1$ and $x, q$ are odd},\\
      1, & \text{otherwise}.
    \end{cases}$$
  \end{thm}
  \noindent Meanwhile, they determined the value of $r^{t- 2, t}_{*}(K_{1, m})$. Moreover, they also determined the value of $r^{2, t}(K_{1, 3})$ and $r_{*}^{2, t}(K_{1, 3})$ for $t\geq 3$.
  
\section{Results and Notations}
  Unfortunately, there is a calculation mistake in Corollary 2.4 in \cite{KO} when $t$ is odd, $s= t- 1$ and $m$ is large. In this paper, under some conditions, we determine the value of $r^{s, t}(K_{1, m_{1}}, \dots, K_{1, m_{c}})$. As corollaries, we completely determine the value of $r^{t- 1, t}(K_{1, m_{1}}, \dots, K_{1, m_{t}})$ and $r^{s, t}(K_{1, m})$ for all $1\leq s< t$. 
  
  For given integers $1\leq s< t$, let $c= {t\choose s}$ and let $A_{1}, A_{2}, \dots, A_{c}$ be subsets with size $s$ of $[t]$. For a set $A$ with size at least $s$, let ${A\choose s}$ be the collection of all subsets with size $s$ of $A$. For each $i\in [c]$, let $m_{A_{i}}$ be a positive integer, and let $K_{1, m_{A_{i}}}$ be the star whose edges are colored by colors in the color set $A_{i}$. In this paper, we assume that all elements in color sets are modules in $[t]$, that is, $b\equiv t$ if $b= at$ for some integer $a$ and $b\equiv a'$ if $b= at+ a'$, where $0< a'< t$. For a given positive integer $a$ and sets $B_{1}, B_{2}, \dots, B_{a}$, let $B_{1}\cup B_{2}\cup \cdots \cup B_{a}$ be the union of $B_{1}, B_{2}, \dots, B_{a}$ with repeat. Moreover, if $B_{1}= B_{2}= \cdots= B_{a}= B$, then we simplify it to $aB$. For example, if $B_{1}= \{1, 2, 3\}$ and $B_{2}= \{1, 2, 3\}$, then $B_{1}\cup B_{2}= \{1, 1, 2, 2, 3, 3\}= 2\{1, 2, 3\}$. For each $i\in [t]$, let
  $$\begin{aligned}
    \ell_{i}= & \left\lfloor \frac{1}{s}\left(\sum_{B\in {[s+ i]\setminus [i] \choose s- 1}} m_{\{i\}\cup B} - (s-1)m_{[s+i]\setminus[i]}- 1\right)\right\rfloor\\
    = & \frac{1}{s}\left(\sum_{B\in {[s+ i]\setminus [i] \choose s- 1}} m_{\{i\}\cup B} - (s-1)m_{[s+i]\setminus[i]}- a_{i}\right).
  \end{aligned}$$
  Note that $a_{i}\in [s]$. In fact, we will prove that $a_{i}= a$ for every $i\in [t]$. Now, we can show our main result.
  \begin{thm}\label{theo1.5}
    For given positive integers $1\leq s< t$, let $c= {t\choose s}$. Suppose that $\sum_{i= 1}^{b} m_{A_{p_{i}}}= \sum_{i= 1}^{b} m_{A_{q_{i}}}$ if $\bigcup_{i= 1}^{b} A_{p_{i}}= \bigcup_{i= 1}^{b} A_{q_{i}}$. Let $k$ be the number of odd numbers of $\{\ell_{1}, \ell_{2}, \dots, \ell_{t}\}$. Then
    $$r^{s, t}\left(K_{1, m_{A_{1}}}, K_{1, m_{A_{2}}}, \dots, K_{1, m_{A_{c}}}\right)= 
	\begin{cases}
		\sum_{i= 1}^{t} \ell_{i}+ a, & \text{$a= 1$ and $k\geq 1$ is even},\\
		\sum_{i= 1}^{t} \ell_{i}+ a+ 1, & \text{otherwise}.
	\end{cases}$$
  \end{thm}
  
  Assume that $s= t- 1$. We can show that if $\bigcup_{i= 1}^{b} A_{p_{i}}= \bigcup_{i= 1}^{b} A_{q_{i}}$, then $\sum_{i= 1}^{b} m_{A_{p_{i}}}= \sum_{i= 1}^{b} m_{A_{q_{i}}}$. Thus, the condition of Theorem \ref{theo1.5} holds.
  \begin{coro}
    For a given positive integer $t$, let $k$ be the number of odd numbers of $\{\ell_{1}, \ell_{2}, \dots, \ell_{t}\}$. Then
    $$r^{t- 1, t}\left(K_{1, m_{A_{1}}}, K_{1, m_{A_{2}}}, \dots, K_{1, m_{A_{t}}}\right)= 
	\begin{cases}
		\sum_{i= 1}^{t} \ell_{i}+ a, & \text{$a= 1$ and $k\geq 1$ is even},\\
		\sum_{i= 1}^{t} \ell_{i}+ a+ 1, & \text{otherwise}.
	\end{cases}$$
  \end{coro}
  
  Note that the condition of Theorem \ref{theo1.5} holds if $K_{1, m_{A_{1}}}= K_{1, m_{A_{2}}}= \cdots= K_{1, m_{A_{c}}}= K_{1, m}$. Thus, the following holds.
  \begin{coro}\label{coro}
    Let $m, k$ and $t$ be positive integers. Let $s$ and $a'$ be integers such that $s\in [t- 1]$ and $0\leq a'\leq s- 1$. Then
    $$r^{s, t}(K_{1, m})= \begin{cases}
      (2k- 1)t+ s+ 1, &  m= 2ks \text{ and } s\neq 1,\\
      2kt+ a'+ 1, & m= 2ks+ a' \text{ and } 1\leq a'\leq s- 1,\\
      2kt+ s+ 1, & m= (2k+ 1)s,\\
      (2k+ 1)t+ 1, & m= (2k+ 1)s+ 1 \text{ and } t \text{ is even},\\
      (2k+ 1)t+ 2, & m= (2k+ 1)s+ 1 \text{ and } t \text{ is odd},\\
      (2k+ 1)t+ a'+ 1, & m= (2k+ 1)s+ a' \text{ and } 2\leq a'\leq s- 1.
    \end{cases}$$
  \end{coro}
  
  Furthermore, we completely determine the value of $r_{*}^{s, t}(K_{1, m})$ as follows.
  
  \begin{thm}\label{theo1.6}
    Let $m, k$ and $t$ be positive integers. Let $s$ be an integer such that $s\in [t- 1]$. Then
    $$r_{*}^{s, t}(K_{1, m})= \begin{cases}
        2kt+ t/2+ 1, & \text{$m= (2k+ 1)s+ 1$, $t$ is even and $1\leq s\leq t/2$},\\
        2kt+ s+ 1, & \text{$m= (2k+ 1)s+ 1$, $t$ is even and $t/2+ 1\leq s\leq t- 1$},\\
		1, & \text{otherwise}.	
	\end{cases}$$
  \end{thm}
	
  \textbf{Notations and definitions}:  A cycle $C$ of a graph $G$ is {\it Hamiltonian} if $V(C)= V(G)$. The {\it center} of a star $K_{1, m}$ is the vertex of degree $m$. A graph $G$ is {\it $r$-regular} if $d(v)= r$ for each $v\in V(G)$. A spanning subgraph $H$ of $G$ is {\it $r$-factor} if $H$ is $r$-regular. Let $G$ be a $t$-edge-colored graph. For each $i\in [t]$, let $G_{i}$ be the subgraph induced by all edges in color $i$. Moreover, let $d_{i}(v)$ be the degree of $v$ in $G_{i}$.
  
\section{Proof of Theorem \ref{theo1.5}}
  In this section, let $1\leq s< t$ be integers and let $c= {t\choose s}$. Let $A_{1}, A_{2}, \dots, A_{c}$ be subsets with size $s$ of $[t]$. Before we prove Theorem \ref{theo1.5}, we still need some preparation. Firstly, we need the following decomposition of a complete graphs (see Harary's book \cite{H}).
  \begin{thm}[\cite{H}]\label{theo2.1}
    $K_{2n}$ can be decomposed into $(2n- 1)$ edge-disjoint $1$-factors, and $K_{2n+ 1}$ can be decomposed into $n$ edge-disjoint $2$-factors.
  \end{thm}
  
  \begin{fact}\label{fact2.2}
     Suppose that $\sum_{i= 1}^{b} m_{A_{p_{i}}}= \sum_{i= 1}^{b} m_{A_{q_{i}}}$ if $\bigcup_{i= 1}^{b} A_{p_{i}}= \bigcup_{i= 1}^{b} A_{q_{i}}$. Then for any set $A\in {[t]\choose s}$, we have 
    $$\sum_{i\in A} \left(\sum_{B\in {[s+ i]\setminus [i] \choose s- 1}} m_{\{i\}\cup B} - (s-1)m_{[s+i]\setminus[i]}\right) = sm_{A}.$$
  \end{fact}
  \begin{proof}
	Without loss of generality, suppose that $A=\{i_{1}, i_{2}, \dots, i_{s}\}$. Thus,
    $$\begin{aligned}
        & \sum_{i\in A} \left(\sum_{B\in {[s+ i]\setminus [i] \choose s- 1}} m_{\{i\}\cup B} - (s-1)m_{[s+i]\setminus[i]}\right)\\
        = & \sum_{i\in A} \left(\sum_{j= 1}^{s} m_{\{i, i+1,\dots ,i+s\}\setminus \{i+j\}}- (s-1)m_{\{i+1, i+2,\dots,i+s\}}\right)\\
        = & \sum_{\ell= 1}^{s} \left(\sum_{j= 1}^{s} m_{\{i_{\ell}, i_{\ell}+1,\dots ,i_{\ell}+s\}\setminus \{i_{\ell}+j\}}- (s-1)m_{\{i_{\ell}+1, i_{\ell}+2,\dots,i_{\ell}+s\}}\right)\\
        = & \sum_{\ell= 1}^{s} \sum_{j= 1}^{s} m_{\{i_{\ell}, i_{\ell}+1,\dots ,i_{\ell}+s\}\setminus \{i_{\ell}+j\}}- (s-1)\sum_{\ell= 1}^{s}m_{\{i_{\ell}+1, i_{\ell}+2,\dots,i_{\ell}+s\}}\\
        = & sm_{\{i_{1}, i_{2}, \dots ,i_{s}\}}+ (s- 1) \sum_{\ell= 1}^{s} m_{\{i_{\ell}+ 1, i_{\ell}+ 2, \dots ,i_{\ell}+ s\}} - (s- 1) \sum_{\ell= 1}^{s} m_{\{i_{\ell}+ 1, i_{\ell}+ 2, \dots ,i_{\ell}+ s\}}\\
        = & sm_{A}
    \end{aligned}$$
    where the second to the last equality is from our assumption since 
    $$\begin{aligned}
      & \bigcup_{\ell= 1}^{s}\left(\bigcup_{j= 1}^{s} \{i_{\ell}, i_{\ell}+ 1, \dots, i_{\ell}+ s\}\setminus \{i_{\ell}+j\}\right)\\
      = & s\{i_{1}, i_{2}, \dots, i_{s}\}\cup \left((s- 1)\bigcup_{\ell= 1}^{s} \{i_{\ell}+ 1, i_{\ell}+ 2, \dots, i_{\ell}+ s\}\right).
    \end{aligned}$$
    We finish the proof.
  \end{proof}

  \noindent Recall that for any $i\in [t]$,
  $$\begin{aligned}
    \ell_{i}= & \left\lfloor \frac{1}{s}\left(\sum_{B\in {[s+ i]\setminus [i] \choose s- 1}} m_{\{i\}\cup B} - (s-1)m_{[s+i]\setminus[i]}- 1\right)\right\rfloor\\
    = & \frac{1}{s}\left(\sum_{B\in {[s+ i]\setminus [i] \choose s- 1}} m_{\{i\}\cup B} - (s-1)m_{[s+i]\setminus[i]}- a_{i}\right).
  \end{aligned}$$
	
  \begin{fact}\label{fact2.3}
    Suppose that $\sum_{i= 1}^{b} m_{A_{p_{i}}}= \sum_{i= 1}^{b} m_{A_{q_{i}}}$ if $\bigcup_{i= 1}^{b} A_{p_{i}}= \bigcup_{i= 1}^{b} A_{q_{i}}$. Then $a_{i}= a_{j}= a$ for all $i, j\in [t]$. Moreover, $a\in [s]$.
  \end{fact}
  \begin{proof}
    For any set $A\in {[t]\choose s}$, we have $\sum_{i\in A} \ell_{i}= m_{A}- \sum_{i\in A}a_{i}/s$ by Fact \ref{fact2.2}. Note that $\sum_{i\in A}a_{i}/s$ is an integer since $\sum_{i\in A} \ell_{i}$ and $m_{A}$ are integers. Thus, $\sum_{i\in A}a_{i}\equiv 0\pmod s$ for any set $A\in {[t]\choose s}$. Let $i_{1}, i_{2}\in [t]$ be integers such that $i_{1}\neq i_{2}$. Let $B\in {[t]\choose s- 1}$ be a set such that $i_{1}\not\in B$ and $i_{2}\not\in B$ (we can do this since $1\leq s< t$). Note that $\sum_{i\in B} a_{i}+ a_{i_{1}}\equiv 0\pmod s$ and $\sum_{i\in B} a_{i}+ a_{i_{2}}\equiv 0\pmod s$ since $B\cup \{i_{1}\}\in {[t]\choose s}$ and $B\cup \{i_{2}\}\in {[t]\choose s}$. Consequently, $a_{i_{1}}- a_{i_{2}}\equiv 0\pmod s$. Furthermore, $a_{i_{1}}= a_{i_{2}}$ since $a_{i_{1}}, a_{i_{2}}\in [s]$. Thus, we are done by the arbitrariness of $i_{1}$ and $i_{2}$.
  \end{proof}

  \noindent The following is deduced from Fact \ref{fact2.2} and Fact \ref{fact2.3}.
  \begin{coro} \label{coro2.4}
    Suppose that $\sum_{i= 1}^{b} m_{A_{p_{i}}}= \sum_{j= i}^{b} m_{A_{q_{i}}}$ if $\bigcup_{i= 1}^{b} A_{p_{i}}= \bigcup_{i= 1}^{b} A_{q_{i}}$. For any set $A\in {[t]\choose s}$, we have $\sum_{i\in A} \ell_{i}= m_{A}- a$, where $a\in [s]$.
  \end{coro}

 Now, we are ready to prove our main result. We divide Theorem \ref{theo1.5} into three parts. Let $k$ be the number of odd numbers in $\{\ell_{1}, \ell_{2}, \dots ,\ell_{t}\}$.
	
  \begin{thm}\label{theo2.5} 
    Suppose that $\sum_{i= 1}^{b} m_{A_{p_{i}}}= \sum_{i= 1}^{b} m_{A_{q_{i}}}$ if $\bigcup_{i= 1}^{b} A_{p_{i}}= \bigcup_{i= 1}^{b} A_{q_{i}}$. Then
	$$r^{s, t}\left(K_{1, m_{A_{1}}}, K_{1, m_{A_{2}}}, \dots, K_{1, m_{A_{c}}}\right)\geq 
	\begin{cases}
		\sum_{i= 1}^{t} \ell_{i}+ a, & \text{$a= 1$ and $k\geq 1$ is even},\\
		\sum_{i= 1}^{t} \ell_{i}+ a+ 1, & \text{otherwise}.
	\end{cases}$$
  \end{thm}
  \begin{proof}
    We can split our discussion into the following cases since $a\in [s]$.
    
	Case 1. $a= 1$ and $k\geq 1$ is even.
    
    Note that $\sum_{i=1}^{t}\ell_{i}$ is even. Thus, $K_{\sum_{i=1}^{t}\ell_{i}}$ can be decomposed into $\sum_{i=1}^{t}\ell_{i}- 1$ edge-disjoint $1$-factors by Theorem \ref{theo2.1}. Furthermore, we have graphs $G_{1}, G_{2}, \dots, G_{t}$ that satisfy the following. (1) $G_{i}$ and $G_{j}$ are edge-disjoint for all $i\in [t], j\in [t]$ and $i\neq j$. (2) $G_{i}$ is $\ell_{i}$-regular for each $i\in [t-1]$ and $G_{t}$ is $(\ell_{t}- 1)$-regular. Color $G_{i}$ by color $i$ for each $i\in [t]$ and denote the resulting graph by $G$. Note that $G$ is a $t$-edge-colored $K_{\sum_{i=1}^{t}\ell_{i}}$. Let $v\in V(G)$ be a vertex and let $A\in {[t]\choose s}$ be a color set. Note that $\sum_{i\in A} d_{i}(v)\leq \sum_{i\in A} \ell_{i}= m_{A}- a= m_{A}- 1$ by Corollary \ref{coro2.4} since $a= 1$. Consequently, there is no copy of $K_{1, m_{A}}$ whose edges are colored by colors in the color set $A$ for each $A\in {[t]\choose s}$, and thus $r^{s, t}(K_{1, m_{A_{1}}}, K_{1, m_{A_{2}}}, \dots, K_{1, m_{A_{c}}})\geq \sum_{i= 1}^{t} \ell_{i}+ 1= \sum_{i= 1}^{t} \ell_{i}+ a$ since $a= 1$.

    Case 2. $a= 1$ and $k= 0$.
    
    Note that $\ell_{i}$ is even for each $i\in [t]$ and $\sum_{i=1}^{t}\ell_{i}+ a$ is odd. Thus, $K_{\sum_{i=1}^{t}\ell_{i}+ a}$ can be decomposed into $\sum_{i=1}^{t}\ell_{i}/2$ edge-disjoint $2$-factors by Theorem \ref{theo2.1} since $a= 1$. Furthermore, we have graphs $G_{1}, G_{2}, \dots, G_{t}$ that satisfy the following since $\ell_{i}$ is even for each $i\in [t]$. (1) $G_{i}$ and $G_{j}$ are edge-disjoint for all $i\in [t], j\in [t]$ and $i\neq j$. (2) $G_{i}$ is $\ell_{i}$-regular for each $i\in [t]$. Color $G_{i}$ by color $i$ for each $i\in [t]$ and denote the resulting graph by $G$. Note that $G$ is a $t$-edge-colored $K_{\sum_{i=1}^{t}\ell_{i}+ a}$. Let $v\in V(G)$ be a vertex and let $A\in {[t]\choose s}$ be a color set. Note that $\sum_{i\in A} d_{i}(v)= \sum_{i\in A} \ell_{i}= m_{A}- a= m_{A}- 1$ by Corollary \ref{coro2.4} since $a= 1$. Consequently, there is no copy of $K_{1, m_{A}}$ whose edges are colored by colors in the color set $A$ for each $A\in {[t]\choose s}$, and thus $r^{s, t}(K_{1, m_{A_{1}}}, K_{1, m_{A_{2}}}, \dots, K_{1, m_{A_{c}}})\geq \sum_{i= 1}^{t} \ell_{i}+ a+ 1$.

    Case 3. $a\geq 2$ and $\sum_{i=1}^{t}\ell_{i}+ a$ is odd.

    Note that $\sum_{i=1}^{t}\ell_{i}+ a- 1$ is even. Thus, $K_{\sum_{i=1}^{t}\ell_{i}+ a- 1}$ can be decomposed into $\sum_{i=1}^{t}\ell_{i}+ a- 2$ edge-disjoint $1$-factors by Theorem \ref{theo2.1}. Furthermore, we have graphs $G_{1}, G_{2}, \dots, G_{t}$ that satisfy the following. (1) $G_{i}$ and $G_{j}$ are edge-disjoint for all $i\in [t], j\in [t]$ and $i\neq j$. (2) $G_{i}$ is $\ell_{i}$-regular for each $i\in [t- 1]$ and $G_{t}$ is $(\ell_{t}+ a- 2)$-regular. Color $G_{i}$ by color $i$ for each $i\in [t]$ and denote the resulting graph by $G$. Let $u\not\in V(G)$ be a vertex and add all edges between $\{u\}$ and $V(G)$. Color $\ell_{i}$ edges adjacent to $u$ by color $i$ for each $i\in [t- 1]$ and $\ell_{t}+ a- 1$ edges adjacent to $u$ by color $t$. Denote the resulting graph by $H$. Note that $H$ is a $t$-edge-colored $K_{\sum_{i=1}^{t}\ell_{i}+ a}$. Let $v\in V(H)$ be a vertex and let $A\in {[t]\choose s}$ be a color set. Note that $\sum_{i\in A} d_{i}(v)\leq \sum_{i\in A} \ell_{i}+ a- 1= m_{A}- 1$ by Corollary \ref{coro2.4} since $a\geq 2$. Consequently, there is no copy of $K_{1, m_{A}}$ whose edges are colored by colors in the color set $A$ for each $A\in {[t]\choose s}$, and thus $r^{s, t}(K_{1, m_{A_{1}}}, K_{1, m_{A_{2}}}, \dots, K_{1, m_{A_{c}}})\geq \sum_{i= 1}^{t} \ell_{i}+ a+ 1$.

    Case 4. $\sum_{i=1}^{t}\ell_{i}+ a$ is even.

    Note that $K_{\sum_{i=1}^{t}\ell_{i}+ a}$ can be partitioned into $\sum_{i=1}^{t}\ell_{i}+ a- 1$ edge-disjoint $1$-factors by Theorem \ref{theo2.1}. Furthermore, we have graphs $G_{1}, G_{2}, \dots, G_{t}$ that satisfy the following. (1) $G_{i}$ and $G_{j}$ are edge-disjoint for all $i\in [t], j\in [t]$ and $i\neq j$. (2) $G_{i}$ is $\ell_{i}$-regular for each $i\in [t-1]$ and $G_{t}$ is $(\ell_{t}+ a- 1)$-regular. Color $G_{i}$ by color $i$ for each $i\in [t]$ and denote the resulting graph by $G$. Let $v\in V(G)$ be a vertex and let $A\in {[t]\choose s}$ be a color set. Note that $\sum_{i\in A} d_{i}(v)\leq \sum_{i\in A} \ell_{i}+ a- 1= m_{A}- 1$ by Corollary \ref{coro2.4} since $a\geq 1$. Consequently, there is no copy of $K_{1, m_{A}}$ whose edges are colored by colors in the color set $A$ for each $A\in {[t]\choose s}$, and thus $r^{s, t}(K_{1, m_{A_{1}}}, K_{1, m_{A_{2}}}, \dots, K_{1, m_{A_{c}}})\geq \sum_{i= 1}^{t} \ell_{i}+ a+ 1$.
  \end{proof}

  \begin{thm}\label{theo2.6} 
    Suppose that $\sum_{i= 1}^{b} m_{A_{p_{i}}}= \sum_{i= 1}^{b} m_{A_{q_{i}}}$ if $\bigcup_{i= 1}^{b} A_{p_{i}}= \bigcup_{i= 1}^{b} A_{q_{i}}$. If $a= 1$ and $k\geq 1$ is even, then $r^{s, t}(K_{1, m_{A_{1}}}, K_{1, m_{A_{2}}}, \dots, K_{1, m_{A_{c}}})= \sum_{i= 1}^{t} \ell_{i}+ a$.
  \end{thm}
  \begin{proof}
    We just need to show that $r^{s, t}(K_{1, m_{A_{1}}}, K_{1, m_{A_{2}}}, \dots, K_{1, m_{A_{c}}})\leq \sum_{i= 1}^{t} \ell_{i}+ a= N$ by Theorem \ref{theo2.5}. Color $K_{N}$ arbitrarily by $t$ colors and denote the resulting graph by $G$. Let $u\in V(G)$ be a vertex. Moreover, let $d_{i}= d_{i}(u)$. Note that $\sum_{i= 1}^{t} d_{i}= d(u)= N- 1= \sum_{i= 1}^{t} \ell_{i}$ since $a= 1$. By the pigeonhole principle, there is $i_{0}\in [t]$ such that $d_{i_{0}}\geq \ell_{i_{0}}$. Without loss of generality, we may assume that $d_{1}= \ell_{1}+ c$, where $c\geq 0$. Moreover, we may assume that $\sum_{B\in {[t]\backslash \{1\}\choose s- 1}}\sum_{i\in \{1\}\cup B} d_{i}\leq \sum_{B\in {[t]\backslash \{1\}\choose s- 1}} (m_{\{1\}\cup B}- 1)$. Otherwise, there is a copy of $K_{1, m_{\{1\}\cup B}}$ whose edges are colored by colors in the color set $\{1\}\cup B$ for some $B\in {[t]\backslash \{1\}\choose s- 1}$, and we are done. Consequently,
    $$\begin{aligned}
        & \sum_{B\in {[t]\backslash \{1\}\choose s- 1}} (m_{\{1\}\cup B}- 1)\geq \sum_{B\in {[t]\backslash \{1\}\choose s- 1}}\sum_{i\in \{1\}\cup B} d_{i}\\
        = & {t- 1\choose s- 1}d_{1}+ {t- 2\choose s- 2}\sum_{i= 2}^{t} d_{i}= {t- 2\choose s- 1}d_{1}+ {t- 2\choose s- 2}d(u)\\
        = & {t- 2\choose s- 1}\ell_{1}+ {t- 2\choose s- 1}c+ {t- 2\choose s- 2}\sum_{i= 1}^{t} \ell_{i}\\
        = & \sum_{B\in {[t]\backslash \{1\}\choose s- 1}} (m_{\{1\}\cup B}- 1)+ {t- 2\choose s- 1}c,
    \end{aligned}$$
    where the last equality is from Corollary \ref{coro2.4} since $a= 1$. Thus, $c= 0$ and $d_{1}= \ell_{1}$. Consequently, there is $j_{0}\in [t]\backslash \{1\}$ such that $d_{j_{0}}\geq \ell_{i_{0}}$ since $\sum_{i= 2}^{t} d_{i}= \sum_{i= 2}^{t} \ell_{i}$. The same discussion as before, we have $d_{i_{0}}= \ell_{i_{0}}$. Repeat the process, either we have a copy of $K_{1, m_{A}}$ whose edges are colored by colors in the color set $A$ for some $A\in {[t]\choose s}$ or $d_{i}= \ell_{i}$ for each $i\in [t]$. We may assume the latter. Note that the subgraph induced by all edges in color $i$ is $\ell_{i}$-regular for each $i\in [t]$ since $u\in V(G)$ is arbitrarily. Recall that $a= 1$ and $k\geq 1$ is even. Thus, there is an $i'\in [t]$ such that $\ell_{i'}$ is odd. Moreover, $\sum_{i= 1}^{t} \ell_{i}+ a$ is odd. It is impossible since the odd regular graph on odd vertices does not exist by the Handshaking Lemma.
  \end{proof}

  \begin{thm}\label{theo2.7} 
    Suppose that $\sum_{i= 1}^{b} m_{A_{p_{i}}}= \sum_{i= 1}^{b} m_{A_{q_{i}}}$ if $\bigcup_{i= 1}^{b} A_{p_{i}}= \bigcup_{i= 1}^{b} A_{q_{i}}$. If $2\leq a\leq s$, or $k= 0$, or $k$ is odd, then $r^{s, t}(K_{1, m_{A_{1}}}, K_{1, m_{A_{2}}}, \dots, K_{1, m_{A_{c}}})= \sum_{i= 1}^{t} \ell_{i}+ a+ 1$.
  \end{thm}
\begin{proof}
    We just need to show that $r^{s, t}(K_{1, m_{A_{1}}}, K_{1, m_{A_{2}}}, \dots, K_{1, m_{A_{c}}})\leq \sum_{i= 1}^{t} \ell_{i}+ a+ 1= N$ by Theorem \ref{theo2.5}. Color $K_{N}$ arbitrarily by $t$ colors and denote the resulting graph by $G$. Let $u\in V(G)$ be a vertex. Moreover, let $d_{i}= d_{i}(u)$. Note that $\sum_{i= 1}^{t} d_{i}= d(u)= N- 1= \sum_{i=1}^{t}\ell_{i}+ a$. In the following, we will show that $G$ contains a copy of $K_{1, m_{A}}$ whose edges are colored by colors in the color set $A$ for some $A\in {[t]\choose s}$ by using induction in $t$.

    Firstly, let $t= s+ 1$. By the pigeonhole principle, there is $i_{0}\in [s+ 1]$ such that $d_{i_{0}}\geq \ell_{i_{0}}+ 1$ since $a\in [s]$. Without loss of generality, we may assume that $d_{1}= \ell_{1}+ 1+ c$, where $c\geq 0$. Moreover, we may assume that $\sum_{B\in {[s+ 1]\backslash \{1\}\choose s- 1}}\sum_{i\in \{1\}\cup B} d_{i}\leq \sum_{B\in {[s+ 1]\backslash \{1\}\choose s- 1}} (m_{\{1\}\cup B}- 1)$. Otherwise, there is a copy of $K_{1, m_{\{1\}\cup B}}$ whose edges are colored by colors in the color set $\{1\}\cup B$ for some $B\in {[s+ 1]\backslash \{1\}\choose s- 1}$, and we are done. Consequently,
    $$\begin{aligned}
        & \sum_{B\in {[s+ 1]\backslash \{1\}\choose s- 1}} (m_{\{1\}\cup B}- 1)\geq \sum_{B\in {[s+ 1]\backslash \{1\}\choose s- 1}}\sum_{i\in \{1\}\cup B} d_{i}\\
        = & sd_{1}+ (s- 1)\sum_{i= 2}^{s+ 1} d_{i}= d_{1}+ (s- 1)d(u)\\
        = & \ell_{1}+ 1+ c+ (s- 1)\left(\sum_{i= 1}^{s+ 1} \ell_{i}+ a\right)\\
        = & \sum_{B\in {[s+ 1]\backslash \{1\}\choose s- 1}} (m_{\{1\}\cup B}- a)+ (s- 1)a+ 1+ c\\
        = & \sum_{B\in {[s+ 1]\backslash \{1\}\choose s- 1}} (m_{\{1\}\cup B}- 1)- a+ s+ 1+ c\\
        \geq & \sum_{B\in {[s+ 1]\backslash \{1\}\choose s- 1}} (m_{\{1\}\cup B}- 1)+ 1,
    \end{aligned}$$
    where the third to the last inequality comes from Corollary \ref{coro2.4} and the last inequality comes from $a\in [s]$ and $c\geq 0$. This is a contradiction. Consequently, the assertion holds for $t= s+ 1$.
    
    Assume that $t\geq s+ 2$ and the assertion holds for $t- 1$. Let $k'$ be the number of odd numbers in $\ell_{1}, \ell_{2}, \dots, \ell_{t- 1}$ and let $c'= {t- 1\choose s}$. Note that $r^{s, t- 1}(K_{1, m_{A_{1}}}, K_{1, m_{A_{2}}}, \dots, K_{1, m_{A_{c'}}})\leq \sum_{i= 1}^{t- 1} \ell_{i}+ a+ 1$ by Theorem \ref{theo2.6} or the induction hypothesis. Consequently, if $\sum_{i= 1}^{t- 1} d_{i}\geq \sum_{i= 1}^{t- 1} \ell_{i}+ a$, then there is a copy of $K_{1, m_{A'}}$ whose edges are colored by colors in the color set $A'$ for some $A'\in {[t- 1]\choose s}$, and we are done. Thus, to finish the proof, we only need to show that $\sum_{i= 1}^{t- 1} d_{i}\geq \sum_{i= 1}^{t- 1} \ell_{i}+ a$.
    
    By the pigeonhole principle, there is $i_{0}\in [t]$ such that $d_{i_{0}}\leq \ell_{i_{0}}$ since $\sum_{i= 1}^{t} d_{i}= d(u)= \sum_{i=1}^{t}\ell_{i}+ a, a\in [s]$ and $t\geq s+ 2$. Without loss of generality, we may assume that $i_{0}= t$ (if $i_{0}\neq t$, then we can relabel the colors such that $i_{0}= t$). Note that 
    $$\sum_{i= 1}^{t- 1} d_{i}= N- 1- d_{t}\geq N- 1- \ell_{t}= \sum_{i= 1}^{t- 1} \ell_{i}+ a.$$
    Thus, we finish the proof.
  \end{proof}
  
  Combining Theorem \ref{theo2.6} and Theorem \ref{theo2.7}, we obtain Theorem \ref{theo1.5}.
  
\section{Proof of Theorem \ref{theo1.6}}
  
  Note that Theorem \ref{theo1.6} holds for $s= 1$ by Theorem \ref{theo1.2} and holds for $s= t- 1$ by Theorem \ref{theo1.4}. Thus, we may assume that $2\leq s\leq t- 2$ in this section. Firstly, we need a stronger decomposition of a complete graph from Bollab{\'a}s.
  \begin{thm}[\cite{B}]\label{theo4.1}
    $K_{2n+ 1}$ can be decomposed into $n$ edge-disjoint Hamiltonian cycles.
  \end{thm}
  
  Let $m, k$ and $t$ be positive integers. Let $s$ and $a'$ be integers such that $s\in [t- 1]$ and $0\leq a'\leq s- 1$. Let $b$ be an integer as follows.
  $$b= \begin{cases}
    (2k- 1)t+ s, & m= 2ks \text{ and } s\neq 1,\\
    2kt+ a', & m= 2ks+ a' \text{ and } 1\leq a'\leq s- 1,\\
    2kt+ s, & m= (2k+ 1)s,\\
    (2k+ 1)t+ 1, & m= (2k+ 1)s+ 1 \text{ and } t \text{ is odd},\\
    (2k+ 1)t+ a', & m= (2k+ 1)s+ a' \text{ and } 2\leq a'\leq s- 1.
  \end{cases}$$
  
  \begin{lem}\label{lemma4.2}
    Let $G$ be a graph, and let $u\in V(G)$ be a vertex such that $d(u)\geq b$. Then every $t$-coloring of $E(G)$ yields a copy of $K_{1, m}$ with center $u$, whose edges are colored by colors in the color set $A_{i}$ for some $A_{i}\in {[t]\choose s}$.
  \end{lem}
  \begin{proof}
    Let $d_{i}= d_{i}(u)$. Without loss of generality, we may assume that $d_{1}\geq d_{2}\geq \dots \geq d_{t}$. If $\sum_{i= 1}^{s} d_{i}\geq m$, then we are done. Thus, we may assume that $\sum_{i= 1}^{s} d_{i}\leq m- 1$. Furthermore, $d_{s}\leq \lfloor (m- 1)/s\rfloor$ by the pigeonhole principle. Note that $\sum_{i= s+ 1}^{t} d_{i}\geq b- \sum_{i= 1}^{s}d_{i}\geq b- m+ 1$. Thus, $d_{s+ 1}\geq \lceil (b- m+ 1)/(t- s)\rceil$ by the pigeonhole principle again. One can directly verify that 
    $$d_{s}\leq \lfloor (m- 1)/s\rfloor< \lceil (b- m+ 1)/(t- s)\rceil\leq d_{s+ 1}.$$
    It is a contradiction to our assumption.
  \end{proof}

  \begin{thm}\label{theo4.3}
    If $m\neq (2k+ 1)s+ 1$, or $t$ is odd, then $r_{*}^{s, t}(K_{1, m})= 1$.
  \end{thm}
  \begin{proof}
    Note that $r^{s, t}(K_{1, m})= b+ 1$ by Corollary \ref{coro}. Thus, $r_{*}^{s, t}(K_{1, m})\leq 1$ by Lemma \ref{lemma4.2}. Consequently, $r_{*}^{s, t}(K_{1, m})= 1$ since $r_{*}^{s, t}(K_{1, m})\geq 1$ by the definition of $r^{s, t}(K_{1, m})$.
  \end{proof}
  
  \begin{thm}\label{theo4.4}
    If $m= (2k+ 1)s+ 1$ and $t$ is even, then
    $$r_{*}^{s, t}(K_{1, m})= \begin{cases}
      2kt+ t/2+ 1, & 2\leq s\leq t/2,\\
      2kt+ s+ 1, & t/2+ 1\leq s\leq t- 2.
    \end{cases}$$
  \end{thm}
  \begin{proof}
    Note that $r^{s, t}(K_{1, m})= (2k+ 1)t+ 1= N$ by Corollary \ref{coro}. Let $x= 2kt+ t/2+ 1$ for $2\leq s\leq t/2$ and $x= 2kt+ s+ 1$ for $t/2+ 1\leq s\leq t- 2$.
    
    Firstly, we will show that $r_{*}^{s, t}(K_{1, m})\geq x$, that is, there is a $t$-coloring of the edges of $K_{N}- E(K_{1, N- x})$ without copy of $K_{1, m}$ whose edges are colored by colors in the color set $A_{i}$ for each $A_{i}\in {[t]\choose s}$. Note that $(2k+ 1)t+ 1$ is odd since $t$ is even. Thus, $K_{(2k+ 1)t+ 1}$ can be decomposed into $(2k+ 1)t/2$ edge-disjoint Hamiltonian cycles by Theorem \ref{theo4.1}. Denote them by $C_{1}, C_{2}, \dots, C_{(2k+ 1)t/2}$. Let $v\in V(K_{N})$ be a vertex. For each $i\in [t/2]$, let $u_{i}v\in E(C_{i})$ be an edge and let $P_{i}$ be the graph obtained from $C_{i}$ by removing $u_{i}v$. Note that $E(P_{i})$ can be decomposed into two edge-disjoint matchings $M_{2(i- 1)+ 1}$ and $M_{2i}$. Color $C_{k(i- 1)+ 1+ t/2}, C_{k(i- 1)+ 2+ t/2}, \dots, C_{ki+ t/2}$ and $M_{i}$ in color $i$ for each $i\in [t]$. Denote the resulting graph by $H'$. Moreover, color $u_{i}v$ in color $i$ for each $i\in [s- t/2]$ if $t/2+ 1\leq s\leq t- 2$. Denote the resulting graph by $H''$. Note that $H'$ contains no copy of $K_{1, m}$ whose edges are colored by colors in the color set $A_{i}$ for each $A_{i}\in {[t]\choose s}$ since $(2k+ 1)s= m- 1< m$. Moreover, if $t/2+ 1\leq s\leq t- 2$, then $H''$ contains no copy of $K_{1, m}$ whose edges are colored by colors in the color set $A_{i}$ for each $A_{i}\in {[t]\choose s}$ since $(2k+ 1)s= m- 1< m$ and $(2k+ 1)t/2+ 2k(s- t/2)+ s- t/2= (2k+ 1)s= m- 1< m$. Consequently, $r_{*}^{s, t}(K_{1, m})\geq x$.
    
    In the following, we will show that $r_{*}^{s, t}(K_{1, m})\leq x$. Let $v\not\in V(K_{N- 1})$ be a vertex. Add $x$ edges between $\{v\}$ and $V(K_{N- 1})$. Color the edges of the resulting graph arbitrarily by $t$ colors. Denote the resulting graph by $G$. Let $u\in V(G)$ be a vertex and let $d_{i}= d_{i}(u)$. Without loss of generality, we may assume that $d_{1}\geq d_{2}\geq \dots \geq d_{t}$. If $\sum_{i= 1}^{s} d_{i}\geq m$, then we are done. Thus, we may assume that $\sum_{i= 1}^{s} d_{i}\leq m- 1= (2k+ 1)s$. Moreover, $d_{s}\leq 2k+ 1$ by the pigeonhole principle. If $d_{s}\leq 2k$, then $\sum_{i= s+ 1}^{t} d_{i}\leq \sum_{i= s+ 1}^{t} d_{s}\leq 2k(t- s)$. Note that $\sum_{i= 1}^{t} d_{i}= x$, $\sum_{i= 1}^{t} d_{i}= (2k+ 1)t- 1$ or $\sum_{i= 1}^{t} d_{i}= (2k+ 1)t$. Consequently,
    $$x\leq \sum_{i= 1}^{t} d_{i}= \sum_{i= 1}^{s} d_{i}+ \sum_{i= s+ 1}^{t} d_{i}\leq m- 1+ 2k(t- s)= 2kt+ s\leq x- 1.$$
    This is a contradiction. Thus, we may assume that $d_{s}= 2k+ 1$. Moreover, $d_{1}= d_{2}= \cdots= d_{s}= 2k+ 1$ since $\sum_{i= 1}^{s} d_{i}\leq (2k+ 1)s$ and $d_{1}\geq d_{2}\geq \cdots\geq d_{s}$. Consequently, $d_{i}(u)\leq 2k+ 1$ for each vertex $u\in V(G)$ and each $i\in [t]$. Note that for each $i\in [t]$, $2e(G_{i})\leq (2k+ 1)N- 1$ since $N$ is odd (from Corollary 9 in \cite{L}). Consequently,
    $$\begin{aligned}
      & (2k+ 1)^{2}t^{2}+ 2kt= [(2k+ 1)N- 1]t\geq \sum_{i= 1}^{t} 2e(H_{i})\\
      = & 2e(H)= 2{N- 1\choose 2}+ 2x= (2k+ 1)^{2}t^{2}- (2k+ 1)t+ 2x\\
      \geq & (2k+ 1)^{2}t^{2}+ 2kt+ 2.
    \end{aligned}$$
    This is a contradiction.
  \end{proof}

  Combining Theorem \ref{theo4.3} and Theorem \ref{theo4.4}, we obtain Theorem \ref{theo1.6}.
  
%\section*{acknowledgements}
%  We are thankful to the reviewers for reading the manuscript very carefully and giving valuable comments on how to improve the manuscript.
%
%\section*{Conflict of interest}
%The authors declare that they have no conflict of interest.
%
%\section*{Data availability}
%No data was used for the research described in the article.

\acknowledgements
\label{sec:ack}
We are thankful to the reviewers for reading the manuscript very carefully and giving valuable comments on how to improve the manuscript.

\nocite{*}
\bibliographystyle{abbrvnat}
% use the following instead if you encounter problems 
%\bibliographystyle{alpha}
\bibliography{references}

\begin{thebibliography}{17}
\providecommand{\natexlab}[1]{#1}
\providecommand{\url}[1]{\texttt{#1}}
\expandafter\ifx\csname urlstyle\endcsname\relax
  \providecommand{\doi}[1]{doi: #1}\else
  \providecommand{\doi}{doi: \begingroup \urlstyle{rm}\Url}\fi

\bibitem[KO1()]{KO1}


\bibitem[Bollob{\'a}s(2004)]{B}
B.~Bollob{\'a}s.
\newblock \emph{{Extremal Graph Theory}}.
\newblock Reprint of the 1978 original, Dover Publications, 2004.

\bibitem[Budden(2023)]{M}
M.~Budden.
\newblock \emph{{Star-critical Ramsey numbers for graphs}}.
\newblock Springer Cham, 2023.
\newblock \url{https://doi.org/10.1007/978-3-031-29981-0}.

\bibitem[Budden and DeJonge(2022)]{BD}
M.~Budden and E.~DeJonge.
\newblock {Multicolor star-critical Ramsey numbers and Ramsey-good graphs}.
\newblock \emph{Electron. J. Graph Theory Appl.}, 10\penalty0 (1):\penalty0
  51--66, 2022.
\newblock \url{https://doi.org/10.5614/ejgta.2022.10.1.4}.

\bibitem[Budden et~al.(2025)Budden, Moun, and Jakhar]{BMJ}
M.~Budden, M.~Moun, and J.~Jakhar.
\newblock {Star-critical weakened Ramsey numbers}.
\newblock \emph{Theory and Appl. of Graphs}, 12\penalty0 (2):\penalty0 \#4,
  2025.
\newblock \url{https://doi.org/10.20429/tag.2025.120204}.

\bibitem[Burr and Roberts(1973)]{BR}
S.~Burr and J.~Roberts.
\newblock {On Ramsey numbers for stars}.
\newblock \emph{Util. Math.}, 4:\penalty0 217--220, 1973.

\bibitem[Chung and Liu(1978)]{CL}
K.~Chung and C.~Liu.
\newblock {A generalization of Ramsey theory for graphs}.
\newblock \emph{Discrete Math.}, 21\penalty0 (2):\penalty0 117--127, 1978.
\newblock \url{https://doi.org/10.1016/0012-365X(78)90084-5}.

\bibitem[Chung et~al.(1997)Chung, Chung, and Liu]{CCL}
K.~Chung, M.~Chung, and C.~Liu.
\newblock {A generalization of Ramsey theory for graphs - with stars and
  complete graphs as forbidden graphs}.
\newblock \emph{Conger. Numer.}, 19:\penalty0 155--161, 1997.

\bibitem[Harary(1969)]{H}
F.~Harary.
\newblock \emph{{Graph Theory}}.
\newblock Addison-Wesley, 1969.

\bibitem[Harary(1972)]{H2}
F.~Harary.
\newblock {Recent results on generalized Ramsey theory for graphs}.
\newblock In Y.~Alavi, D.~R. Lick, and A.~T. White, editors, \emph{Graph Theory
  and Applications}, pages 125--138, 1972.
\newblock \url{https://doi.org/10.1007/BFb0067364}.

\bibitem[Harborth and M{\"o}ller(1999)]{HM}
H.~Harborth and M.~M{\"o}ller.
\newblock {Weakened Ramsey numbers}.
\newblock \emph{Discrete Appl. Math.}, 95\penalty0 (1-3):\penalty0 279--284,
  1999.
\newblock \url{https://doi.org/10.1016/S0166-218X(99)00081-5}.

\bibitem[Hook(2010)]{H3}
J.~Hook.
\newblock \emph{{The classification of critical graphs and star-critical Ramsey
  numbers}}.
\newblock Phd thesis, Lehigh University, 2010.

\bibitem[Hook(2024)]{J2}
J.~Hook.
\newblock {Recent developments of star-critical Ramsey numbers}.
\newblock In F.~Hoffman, S.~Holliday, Z.~Rosen, F.~Shahrokhi, and J.~Wierman,
  editors, \emph{Combinatorics, Graph Theory, and Computing}, pages 245--254,
  2024.
\newblock \url{https://doi.org/10.1007/978-3-031-52969-6_22}.

\bibitem[Khamseh and Omidi(2017)]{KO1}
A.~Khamseh and G.~Omidi.
\newblock {A generalization of Ramsey theory for stars and one matching}.
\newblock \emph{Math. Reports}, 69\penalty0 (1):\penalty0 85--92, 2017.
\newblock
  \url{https://www.imar.ro/journals/Mathematical_Reports/Pdfs/2017/1/8.pdf}.

\bibitem[Khamseh and Omidi(2022)]{KO}
A.~Khamseh and G.~Omidi.
\newblock {Large stars with few colors}.
\newblock \emph{Math. Reports}, 24\penalty0 (74):\penalty0 363--375, 2022.
\newblock
  \url{https://www.imar.ro/journals/Mathematical_Reports/Pdfs/2022/3/4.pdf}.

\bibitem[Luo(2025)]{L}
Z.~Luo.
\newblock {A generalization of Ramsey theory for stars and one matching}.
\newblock \emph{Discuss. Math. Graph Theory}, 45\penalty0 (2):\penalty0
  755--762, 2025.
\newblock \url{https://doi.org/10.7151/dmgt.2550}.

\bibitem[Radziszowski(2024)]{SR}
S.~Radziszowski.
\newblock {Small Ramsey numbers}.
\newblock \emph{Electron. J. Combin.}, DS1.17:\penalty0 133 pp, 2024.
\newblock \url{https://doi.org/10.37236/21}.

\bibitem[Ramsey(1930)]{FR}
F.~Ramsey.
\newblock {On a problem of formal logic}.
\newblock \emph{Proc. Lond. Math. Soc.}, s2-30\penalty0 (1):\penalty0 264--286,
  1930.
\newblock \url{https://doi.org/10.1112/plms/s2-30.1.264}.

\end{thebibliography}
\label{sec:biblio}

\end{document}